\documentclass[10pt,a4paper]{amsart}

\usepackage{amsmath,amssymb,amsthm}

\DeclareMathOperator{\Ima}{Im}

\newtheorem{Def}{Definition}[section]
\newtheorem{Thm}[Def]{Theorem}
\newtheorem{Lem}[Def]{Lemma}
\newtheorem{Prop}[Def]{Proposition}
\newtheorem{Cor}[Def]{Corollary}

\newtheorem{Conj}[Def]{Conjecture}

\begin{document}

\title[On the modularity of elliptic curves]{On the modularity of elliptic curves over a composite field of some real quadratic fields}
\author{Sho Yoshikawa}
\address{Graduate School of Mathematical Sciences,
the University of Tokyo, 3-8-1, Komaba, Meguro-Ku, Tokyo 153-8914, Japan}
\email{yoshi@ms.u-tokyo.ac.jp}

\maketitle

\begin{abstract}
Let $K$ be a composite field of some real quadratic fields.
We give a sufficient condition on $K$ such that all elliptic curves over $K$ is modular.
\end{abstract}

\section{Introduction}
For an elliptic curve $E$ over a totally real field $K$, we say that $E$ is modular if there exists a Hilbert modular form $f$ over $K$ of parallel weight 2 which satisfies the equality
\[ L(E,s)=L(f,s). \]
The original Shimura-Taniyama conjecture asserts that any elliptic curve over $\mathbb{Q}$ is modular.
This conjecture was proved in the celebrated works by Wiles \cite{Wi}, Taylor-Wiles \cite{TW}, and Breuil-Conrad-Diamon-Taylor \cite{BCDT}.
The Shimura-Taniyama conjecture has a natural generalization to the totally real field setting:
\begin{Conj}
\label{ST}
Let $K$ be a totally real number field. Then, any elliptic curve over $K$ is modular.
\end{Conj}

A number of developments of modularity lifting theorem enable us to prove that many elliptic curves are modular.
However, it had been difficult to prove the modularity of \textit{all} elliptic curves over a fixed field.
A few years ago, Le Hung and Freitas-Le Hung-Siksek brought a breakthrough in the problem of modularity of elliptic curves:

\begin{Thm}
(\cite[Theorem 1]{FLS})
\label{quad}
Let $K$ be a real quadratic number field. Then, any elliptic curve over $K$ is modular.
\end{Thm}

Also, using the results in \cite{Th} and in Iwasawa-theory for elliptic curves, Thorne recently proved the following theorem:

\begin{Thm}
(\cite[Theorem 1]{Th2})
Let $p$ be a prime number and $K$ be a totally real field contained in a $\mathbb{Z}_p$-cyclotomic extension of $\mathbb{Q}$.
Then, any elliptic curve over $K$ is modular.
\end{Thm}

In this paper, we prove some special cases of Conjecture \ref{ST} by making use of methods in \cite{FLS}.
Before stating our main theorem, we introduce a notation and give a remark:
For an elliptic curve $X$ over a field $F$, a Galois extension $K/F$, and a character $s: \mathrm{Gal}(K/F)\rightarrow \{ \pm 1\}$,  we write $X^{(s)}$ for the quadratic twist of $X$ by $\tilde{s}: \mathrm{Gal}(\bar{F}/F)\rightarrow \mathrm{Gal}(K/F)\stackrel{s}{\longrightarrow} \{ \pm 1\}$. 
Also, we note that the modular curve $X_0(15)$ (resp. $X_0(21)$) is an elliptic curve of rank 0 with Cremona label 15A1 (resp. 21A1);
for example, see \cite{FLS} (Magma scripts are available at http://arxiv.org/abs/1310.7088).
The main theorem of this paper is then the following:
\begin{Thm}
\label{main}
Let $p=5$ or $7$, and $X$ be the modular curve $X_0(3p)$.
Let $K$ be a composite field of finite number of real quadratic fields.
We assume that $K$ is unramified at every prime dividing $6p$.
We furthermore assume that, for each character $s: \mathrm{Gal}(K/\mathbb{Q})\rightarrow \{ \pm 1\}$ , the Mordell-Weil group $X^{(s)}(\mathbb{Q})$ is of rank 0.
Then, any elliptic curve over $K$ is modular.
\end{Thm}

In contrast to the case of cyclic field extensions as considered in \cite{Th}, we consider here certain extensions which are far from cyclic ones.
We prove Theorem \ref{main} in the next section. It is motivated from the proof of \cite[Lemma 1.1]{FLS} and is outlined as follows:
By \cite[Theorem 2]{FLS}, an elliptic curve which is not yet proved to be modular defines a point of a certain modular curve.
We show that the point of the modular curve is actually a rational point or a real quadratic point, both of which corresponds to a modular elliptic curve.

After the proof of Theorem \ref{main}, we also discuss how often the hypothesis of Theorem \ref{main} is likely to hold.

\section{Proof of Theorem \ref{main}}

For a group $G$, a $\mathbb{Z}[G]$-module $M$, and a character $s: G\rightarrow \{ \pm 1\}$,
we write $M_s$ for the subgroup of $M$ defined by
\[ M_s=\{ m\in M; m^\sigma =s_\sigma m \mathrm{\ for\ all\ } \sigma \in G\}. \]
The following lemma immediately follows from the definition of quadratic twists.
 
\begin{Lem}
\label{tMW}
Let $K/F$ be a Galois extension.
Let $X$ be an elliptic curve over $F$.
Then, for each character $s: \mathrm{Gal}(K/F)\rightarrow \{ \pm 1\}$, we have
\[ X(K)_s\simeq X^{(s)}(F). \]
\end{Lem}

\begin{proof}
By the definition of quadratic twists, 
we have an isomorphism $f: X\stackrel{\simeq}{\longrightarrow} X^{(s)}$ over $K$ which satisfies $f(P^\sigma)=s_\sigma f(P)^\sigma$ for $P\in X(K)$ and $\sigma\in G=\mathrm{Gal}(K/F)$.
By the isomorphism $f$, the subgroup $X(K)_s\subset X(K)$ corresponds to the subgroup $X^{(s)}(F)\subset X^{(s)}(K)$.
\end{proof}

For a group $G$ isomorhic to $(\mathbb{Z}/(2))^r$ with $r$ a positive integer, let $G^{\vee}$ denotes here the group $\mathrm{Hom}(G, \{ \pm 1\})$ of characters.

\begin{Lem}
\label{sum}
Let $G$ is a group isomorphic to $(\mathbb{Z}/(2))^r$ for a positive integer $r$ and $M$ is a $\mathbb{Z}[G]$-module, then we have $2^r M\subset \sum_{s\in G^{\vee}} M_s$.
\end{Lem}

\begin{proof}
This is clear by noting that, for $m\in M$, $2^rm$ is written as
\[2^rm= \sum_{s\in G^{\vee}} \sum_{\sigma \in G} s_\sigma m^\sigma,\]
and that the element $\sum_{\sigma \in G} s_\sigma m^\sigma$ belongs to $M_s$.
\end{proof}

For an elliptic curve $E$ over a field $K$ and a prime number $\ell$, let $\bar{\rho}_{E,\ell}: \mathrm{Gal}(\bar{K}/K)\rightarrow \mathit{GL}_2(\mathbb{F}_\ell)$ denote the mod $\ell$ Galois representation attached to the $\ell$-torsion points of $E$.

\begin{Lem}
\label{surj}
For $X=X_0(15)$ or $X_0(21)$ and a prime number $\ell\geq 3$, the mod $\ell$ Galois representation $\bar{\rho}_{X,\ell}$ is surjective. 
\end{Lem}

\begin{proof}
Recall that $X_0(15)$ (resp. $X_0(21)$) is the elliptic curve with Cremona label 15A1 (resp. 21A1).
The $j$-invariant $j_X$ of $X$ is given by
\[
j_X = 
\begin{cases}
     3^{-4}\cdot 5^{-4}\cdot  13^3\cdot 37^3 & (\mathrm{if}\, X=X_0(15)) \\
     3^{-4}\cdot 7^{-2}\cdot 193^3 & (\mathrm{if}\, X=X_0(21)).
 \end{cases}
\]
Thus, $\ell$ does not divide the exponents of $3$ and $5$ (resp. $3$ and $7$) in $j_{X_0(15)}$ (resp. $j_{X_0(21)}$). 
Also, by hand or by looking at the coefficients of the modular form corresponding to $X$, it is checked that $|X_0(15)(\mathbb{F}_7)|=|X_0(21)(\mathbb{F}_5)|=8$; in particular, these are not divisible by $\ell$.
Applying \cite[Proposition 21]{Se} to our $X$ and $\ell$, we see that $\bar{\rho}_{X,\ell}$ is surjective.
\end{proof}

Using the above two basic results, we prove the following proposition.

\begin{Prop}
\label{descent}
Under the assumption of Theorem \ref{main},
we have $X(K)=X(\mathbb{Q})$.
\end{Prop}

\begin{proof}
Let $G$ denote the Galois group $\mathrm{Gal}(K/\mathbb{Q})$, which is by assumption isomorphic to $(\mathbb{Z}/(2))^r$ for a positive integer $r$.
By Lemma \ref{surj}, $\bar{\rho}_{X,\ell}$ for every prime $\ell\geq 3$ is in particular irreducible, and thus $X^{(s)}(\mathbb{Q})$ for $s\in G^{\vee}$ has only 2-power torsion points.
Also, for each $s\in G^{\vee}$, $X^{(s)}(\mathbb{Q})$ is assumed to be of rank 0, and Lemma \ref{tMW} implies that $X(K)_s\simeq X^{(s)}(\mathbb{Q})$.
It follows that all $X(K)_s$ for $s\in G^{\vee}$ are killed by $[2^n]: X\rightarrow X$ for some positive integer $n$.
Then, since $[2^r] X(K)\subset \sum_{s\in G^{\vee}}X(K)_s$ by Lemma \ref{sum},
we have $[2^{n+r}] X(K)=0$; that is, $X(K)\subset X[2^{n+r}](\bar{\mathbb{Q}})$.
This implies that $G$-module $X(K)$ can be ramified only at primes dividing $6p$.
On the other hand, by the assumption that $K$ is unramified at primes dividing $6p$, it follows that $X(K)$ is unramified everywhere.
Therefore, we have $X(K)=X(\mathbb{Q})$. 
\end{proof}

We note the following modularity theorem, which is a consequence of a general modularity lifting theorem, the theorem of Langlands-Tunnel, and the modularity switching arguments. 

\begin{Thm}
(\cite[Theorem 2]{FLS})
\label{mlt}
Let $E$ be an elliptic curve over a totally real field $K$.
If $p=3,5,$ or $7$, and if the mod $p$ Galois representation $\bar{\rho}_{E,p}|_{G_{K(\zeta_p)}}$ is absolutely irreducible,
then $E$ is modular.
\end{Thm}

For mod $5$ (resp. mod $7$) representations, we can relax the hypothesis in Theorem \ref{mlt} under the assumption that the base field is unramified at $5$ (resp. $7$).

\begin{Thm}
\label{Irr}
Let $p=5$ or $7$ and $K$ be a totally real field unramified at $p$.
If $E$ is an elliptic curve over $K$ such that the mod $p$ Galois representation $\bar{\rho}_{E,p}$ is (absolutely) irreducible,
then $E$ is modular. 
\end{Thm}
\begin{proof}
The case $p=5$ is proved in \cite{Th} by Thorne. He in fact obtains a slightly stronger result: we only need to assume that $\sqrt{5}\notin K$ instead of the unramifiedness of 5 in $K$. The case $p=7$ is proved in \cite{Yo} by the author.
\end{proof}

Before proceeding to the proof of Theorem \ref{main}, we also need to introduce a certain modular curve from \cite{FLS}.
For a prime number $p\neq 3$, let $X(s3, bp)$ denote the modular curve classifying elliptic curves such that $\Ima \bar{\rho}_{E,3}$ is contained in the nomalizer of a split Cartan subgroup of $\mathit{GL}_2(\mathbb{F}_3)$ and that $\bar{\rho}_{E,p}$ is reducible.
For the details of such a modular curve, we refer the reader to \cite{FLS}.
For the proof of Theorem \ref{main}, we only need the following properties of $X(s3, b5)$ and $X(s3, b7)$.

\begin{Lem}
\label{s3b5}
The following are true:
\begin{enumerate}
\item The modular curve $X(s3, b5)$ is an elliptic curve over $\mathbb{Q}$ and $X(s3, b5)$ is isogeneous to $X_0(15)$ by an isogeny of degree 2.
\item The modular curve $X(s3, b7)$ is isomorphic to $X_0(63)/\langle w_9\rangle$, and  the curve $X_0(63)/\langle w_7, w_9\rangle$ is isomorphic to $X_0(21)$, where $w_7$ and $w_9$ are the Atkin-Lehner involutions on $X_0(63)$. In particular, $X(s3, b7)$ admits a morphism to $X_0(21)$ of degree 2.
\end{enumerate}
\end{Lem}

\begin{proof}
See \cite[Proposition 4]{Th2} for (i), and see \cite[Proof of Lemma 1.1]{FLS} for (ii).
\end{proof}

We are now ready to prove Theorem \ref{main}.
\vspace{0.3cm}\\
\textit{Proof of the Theorem \ref{main}.\, }
Let $p$, $X$, and $K$ be as in the statement of Theorem \ref{main}.

Suppose we are given an elliptic curve $E$ over $K$. We show that $E$ is modular.
Because of Theorem \ref{mlt} and Theorem \ref{Irr}, we have only to consider the case where $\bar{\rho}_{E,3}|_{G_{K(\zeta_3)}}$ is absolutely reducible and $\bar{\rho}_{E,p}$ is reducible.
In such a case, $E$ defines a $K$-point of $X$ or $X(s3, bp)$ by \cite[Proposition 4.1, Corollary 10.1]{FLS};

Suppose first that $E$ defines a $K$-point of $X$.
Then Proposition \ref{descent} implies that, after a suitable base change by a solvable Galois extension, $E$ becomes actually defined over $\mathbb{Q}$, and hence $E$ is modular (by \cite[Theorem A]{BCDT} and the solvable base change theorem).

Next we consider the case where $E$ defines a $K$-point $P=P_E$ of $X(s3, bp)$.
By Lemma \ref{s3b5}, we have a morphism $f_p: X(s3, bp)\rightarrow X$ of degree 2.
By Proposition \ref{descent}, $\mathrm{Gal}(K/\mathbb{Q})$ acts on the fiber $f_p^{-1}(f_p(P))$.
Since $f_p^{-1}(f_p(P))$ consists of 2 points,
the kernel of this action is a subgroup in $\mathrm{Gal}(K/\mathbb{Q})$ of index at most $2$. 
It follows that $P$ must be a rational point or a real quadratic point of $X$.
Similarly to the previous paragraph, \cite[Theorem A]{BCDT}, \cite[Theorem 1]{FLS}, and the base change theorem show that $E$ is modular. 
This completes the proof. \hfill $\Box$\vspace{0.3cm}

Let again the notation be as in Theorem \ref{main}.
We discuss here how many totally real fields $K$ are expected to satisfy the hypothesis of Theorem \ref{main}:
We expect by intuition that, for each positive integer $r$, there are infinitely many $K$ with $[K:\mathbb{Q}]=2^r$ satisfying the condition of Theorem \ref{main}.
To explain this, we first note the following theorem on the description of local root numbers of an elliptic curve.
For an elliptic curve $E$ over a local or global field $K$, we write $w(E/K)$ for the root number of $E$.

\begin{Thm}
\label{root}
\cite[Theorem 3.1]{DD}
Let $E$ be an elliptic curve over a local field $K_v$. Then,
\begin{itemize}
\item[(1)] $w(E/K_v)=-1$ if $v|\infty$ or $E$ has split multiplicative reduction.
\item[(2)] $w(E/K_v)=1$ if $E$ has either good or nonsplit multiplicative reduction.
\item[(3)] $w(E/K_v)=\large( \frac{-1}{k}\large) $ if $E$ has additive potentially multiplicative reduction, and the residue field $k$ of $K_v$ has characteristic $p\geq 3$. Here, $\large( \frac{-1}{k}\large)=1$ (resp. $-1$) if $-1\in (k^{\times})^2$ (resp. otherwise).
\item[(4)] $w(E/K_v)=(-1)^{\lfloor \mathrm{ord}_v(\Delta)|k|/12 \rfloor}$, if $E$ has potentially good reduction, and the residue field $k$ of $K_v$ has characteristic $p\geq$ 5. Here, $\Delta$ is the minimal discriminant of $E$, and $\lfloor x \rfloor$ is the greatest integer $n\leq x$.
\end{itemize}
\end{Thm}

%

Using Theorem \ref{root}, we calculate in an elementary way the global root numbers of quadratic twists of an elliptic curve that we are interested in. 

\begin{Cor}
\label{twistroot}
Let $E$ be a semi-stable elliptic curve over $\mathbb{Q}$ with the odd conductor $N$, and $d\equiv 1$ mod $4$ be a square-free positive integer prime to $N$.
Then, we have  $w(E^{(d)}/\mathbb{Q}) = \bigl( \frac{d}{N}\bigr) w(E/\mathbb{Q})$, where $E^{(d)}$ denotes the quadratic twist of $E$ by $d$ and  $\bigl( \frac{\cdot}{N} \bigr)$ is the Jacobi symbol.
\end{Cor}

\begin{proof}
Note that $E^{(d)}$ has the conductor $d^2N$ because $d\equiv 1$ mod 4.
Let $p$ be a prime number.
\begin{itemize}
\item If $p\nmid dN$, then both $E^{(d)}$ and $E$ has good reduction at $p$ and thus $w(E/\mathbb{Q}_p)=w(E^{(d)}/\mathbb{Q}_p)=1$ by Theorem \ref{root} (2).
\item If $p| d$, then $E^{(d)}$ has additive potentially good reduction at $p$ and it acquires good reduction over the quadratic extension $\mathbb{Q}_p(\sqrt{d})$ of $\mathbb{Q}_p$. Thus, by Theorem \ref{root} (4), $w(E^{(d)}/\mathbb{Q}_p)=(-1)^{\lfloor p/2\rfloor}$, which is equal to $1$ (resp. $-1$) for $p\equiv 1$ mod $4$ (resp. $p\equiv 3$ mod $4$).
\item Suppose here that $p|N$. Thus, $E$ and $E^{(d)}$ have multiplicative reduction at $p$.
Take a minimal Weierstrass equation
\[ y^2=f(x) \]
of $E$ over $\mathbb{Q}_p$, where $f(x)\in \mathbb{Z}_p[x]$ is a monic polynomial of degree 3.
Write $\mathbb{F}_p$ as the union of the subsets
\begin{eqnarray*}
S_0&=&\{ x\in \mathbb{F}_p; f(x)=0 \} \\
S^+&=&\{ x\in \mathbb{F}_p; \Big( \frac{f(x)}{p} \Big)=1 \}\\
S^-&=&\{ x\in \mathbb{F}_p; \Big( \frac{f(x)}{p} \Big)=-1 \}.\\
\end{eqnarray*}
In particular, we have $|E(\mathbb{F}_p)|=|S_0|+2|S^+|+1$.
Note that, for an elliptic curve $X$ over $\mathbb{Q}_p$ with bad reduction, 
\[
  p+1-|X(\mathbb{F}_p)| = \begin{cases}
    0 & (\text{if $X$ has additive reduction}) \\
    1  & (\text{if $X$ has split multiplicative reduction})\\
    -1 & (\text{if $X$ has non-split multiplicative reduction})
  \end{cases}
\]
If $\big( \frac{d}{p} \big)=1$, then we have $|E^{(d)}(\mathbb{F}_p)|=|S_0|+2|S^+|+1=|E(\mathbb{F}_p)|$, and hence $E$ has (non-)split  multiplicative reduction if and only if so does $E^{(d)}$.
If $\big( \frac{d}{p} \big)=1$, then we have $|E^{(d)}(\mathbb{F}_p)|=|S_0|+2|S^-|+1=2p+2-|E(\mathbb{F}_p)|$, and hence $E$ has split (resp. non-split)  multiplicative reduction if and only if $E^{(d)}$ has non-split (resp. split) multiplicative reduction.
Summarizing these arguments and Theorem \ref{root} (1), (2), we obtain $w(E^{(d)}/\mathbb{Q}_p)=\big( \frac{d}{p}\big)w(E/\mathbb{Q}_p)$ for every $p|N$.
\item Also, $w(E/\mathbb{R})=w(E^{(d)}/\mathbb{R})=-1$ by Theorem \ref{root} (1).
\end{itemize}
Taking the products of the local root numbers over all places of $\mathbb{Q}$, we obtain the desired formula.
\end{proof}

Let $r\geq 2$ be an integer and $d_1$,...,$d_r$ be square-free positive integers satisfying the following conditions:
\begin{itemize}
\item $(d_1,...,d_r, 3p)=1$, 
\item $d_i\equiv 1$ mod $4$ for $i=1,...,r$, and
\item $\big( \frac{d_i}{3p} \big)=1$ for $i=1,...,r$.
\end{itemize}
Here, recall that $3p$ is the conductor of $X$, and note that there infinitely many choices of such $r$-tuples $(d_1,...,d_r)$.
The field $K:=\mathbb{Q}(\sqrt{d_1},...,\sqrt{d_r})$ is unramified at every prime dividing $6p$.
Also, Corollary \ref{twistroot} shows that $w(X^{(s)}/\mathbb{Q})=w(X/\mathbb{Q})$ for any $s: \mathrm{Gal}(K/\mathbb{Q})\rightarrow \{ \pm 1\}$.
Since $\mathrm{rank}\, X =0$, the parity conjecture for our $X$ therefore predicts that $\mathrm{rank}\, X^{(s)}$ for any $s$ is even. 
The Goldfeld conjecture \cite{G} suggests that, for an elliptic curve over $\mathbb{Q}$, most of its quadratic twists of even (resp. odd) rank would be of rank 0 (resp. 1).
Thus, it seems reasonable to expect that the fields $K=\mathbb{Q}(\sqrt{d_1},...,\sqrt{d_r})$ for most $(d_1,...,d_r)$ satisfy the hypothesis in Theorem \ref{main}, although the two conjectures does \textit{not} imply that this is actually true. 
\section*{Acknowledgement}
The author is especially grateful to his advisor Professor Takeshi Saito for a number of helpful suggestions to improve the arguments in this paper.
He also thanks Professor Ken Ono for informing the author of some references on the rank of quadratic twists.

\end{document}